\documentclass[11pt,english,a4paper]{article}
\usepackage[centertags]{amsmath}
\usepackage{amssymb}
\usepackage{amsthm}
\usepackage{makeidx}

\usepackage[margin=1.5in]{geometry}
\usepackage{float}
\usepackage{lscape}
\usepackage{graphicx}
\usepackage{amsmath}
\usepackage[table]{xcolor}

\newtheorem{theorem}{Theorem}[section]

\newtheorem{lemma}[theorem]{Lemma}

\newtheorem{definition}[theorem]{Definition}
\newtheorem{remark}[theorem]{Remark}

\numberwithin{equation}{section} 

\begin{document}

\title{About the (Hahn) classical character of $2$-orthogonal solutions of two families of differential equations of third order}

\author{T. A. Mesquita\footnote{Corresponding author (taugusta.mesquita@gmail.com)} }
\date{}
\maketitle

\begin{center}
{\scriptsize
Escola Superior de Tecnologia e Gest\~ao, Instituto Polit{\'e}cnico de Viana do Castelo, Rua Escola Industrial e Comercial de Nun' \'Alvares, 4900-347, Viana do Castelo, Portugal, \&\\
Centro de Matem\'{a}tica da Universidade do Porto, Rua do Campo Alegre, 687, 4169-007 Porto, Portugal\\
}
\end{center}

\begin{abstract}
Considering a differential operator of third order that does not increase the degree of polynomials, we analyse some properties of elements of the dual space of $2$-orthogonal polynomial eigenfunctions.
In two classes of such generic operator, we prove that a $2$-orthogonal polynomial solution fulfils Hahn's property.
\end{abstract}

 \textbf{Keywords and phrases}: $d$-orthogonal polynomials,\, differential operators, \, two-orthogonal polynomials,\, Hahn's property.\\
 \textbf{2010 Mathematics Subject Classification}: 42C05 \,,\,  33C45 \,,\, 33D45

\section{Basic definitions and notation}
Let $\mathcal{P}$ be the vector space of polynomials with coefficients in $\mathbb{C}$ and let
$\mathcal{P}^{\prime}$ be its topological dual space. We denote by $\langle u ,p\rangle$ the action of the form or linear functional $u \in\mathcal{P}^{\prime }$ on $p\in\mathcal{P}$. In particular,
$\langle u, x^{n}\rangle:=\nolinebreak\left( u \right) _{n},n\geq 0$ represent the moments of $u$.
In the following, we will call polynomial sequence (PS) to any sequence
${\{P_{n}\}}_{n \geq 0}$ such that $\deg P_{n}= n,\; \forall n \geq 0$.
We will also call monic polynomial sequence (MPS) to a PS so that all polynomials have leading coefficient equal to one.

If ${\{P_{n}\}}_{n \geq 0}$ is a MPS, there exists a unique sequence
$\{u_n\}_{n\geq 0}$, $u_n\in\mathcal{P}^{\prime}$, called the dual sequence of $\{P_{n}\}_{n\geq 0}$, such that,
\begin{equation}\label{SucDual}
<u_{n},P_{m}>=\delta_{n,m}\ , \ n,m\ge 0.
\end{equation}

%
\noindent On the other hand, given a MPS ${\{P_{n}\}}_{n \geq 0}$, the expansion of $xP_{n+1}(x)$, defines sequences in $\mathbb{C}$, ${\{\beta_{n}\}}_{n \geq 0}$ and
$\{\chi_{n,\nu}\}_{0 \leq \nu \leq n,\; n \geq 0},$ such that
\begin{eqnarray}
&&P_{0}(x)=1, \;\; P_{1}(x)=x-\beta_{0},\label{divisao_ci}\\
 && xP_{n+1}(x)= P_{n+2}(x)+ \beta_{n+1}P_{n+1}(x)+\sum_{\nu=0}^{n}\chi_{n,\nu}P_{\nu}(x).\label{divisao}
\end{eqnarray}
This relation is usually called the structure relation of  ${\{P_{n}\}}_{n \geq 0}$, and ${\{\beta_{n}\}}_{n \geq 0}$ and
$\{\chi_{n,\nu}\}_{0 \leq \nu \leq n,\; n \geq 0}$ are called the structure coefficients (SCs) \cite{theoriealgebrique}.
\par In this paper, we will deal with $2$-orthogonal MPSs, hence it is important to mention in this section the definition of $d$-orthogonal MPSs (where $d$ is a positive integer) along with basic definitions, in particular regarding the differential operator $D$.
\medskip

\begin{definition}\cite{Douak2classiques,Maroni-Toulouse,Van-Iseg-1987}
Given $\Gamma^{1}, \Gamma^{2},\ldots, \Gamma^{d} \in \mathcal{P}^{\prime}$, $d \geq 1$, the polynomial sequence ${\{P_{n}\}}_{n \geq 0}$ is called d-orthogonal polynomial sequence (d-OPS) with respect to  $\Gamma=(\Gamma^{1},\ldots, \Gamma^{d} )$ if it fulfils 
\begin{equation}\label{d ortogonal}
 \langle \Gamma^{\alpha}, P_{m}P_{n}  \rangle=0, \;\; n \geq md +\alpha,\,\; m \geq 0,
\end{equation}
 \begin{equation}\label{d ortogonal regular}
 \langle \Gamma^{\alpha}, P_{m}P_{md+\alpha-1}  \rangle \neq 0, \;\; m \geq 0,
\end{equation}
for each integer $\alpha=1,\ldots, d$.

%
\begin{lemma}\cite{variations}\label{lema1}
For each $u \in \mathcal{P}^{\prime}$ and each $m \geq 1$, the two following propositions are equivalent.
\begin{description}
\item[a)]  $\langle u, P_{m-1} \rangle \neq 0,\:\: \langle u, P_{n} \rangle=0,\: n \geq m$.
\item[b)] $\exists \lambda_{\nu} \in \mathbb{C},\:\: 0 \leq \nu \leq m-1,\:\: \lambda_{m-1}\neq 0$ such that
$u=\sum_{\nu=0}^{m-1}\lambda_{\nu} u_{\nu}$. \end{description}
\end{lemma}

The conditions (\ref{d ortogonal}) are called the d-orthogonality conditions and the conditions (\ref{d ortogonal regular}) are called the regularity conditions. In this case, the functional $\Gamma$, of dimension $d$, is said regular.
\end{definition}
The  $d$-dimensional functional $\Gamma$ is not unique. Nevertheless, from Lemma \ref{lema1}, we have:
$$\Gamma^{\alpha}=\sum_{\nu=0}^{\alpha-1}\lambda_{\nu}^{\alpha} u_{\nu}, \;\; \lambda_{\alpha-1}^{\alpha} \neq 0, \; 1\leq \alpha \leq d.$$
Therefore, since $U=(u_{0},\ldots, u_{d-1} )$ is unique, we use to consider the canonical functional of dimension $d$, $U=(u_{0},\ldots, u_{d-1} )$, saying that
${\{P_{n}\}}_{n \geq 0}$ is d-OPS ($d\geq 1$) with respect to
$U=(u_{0},\ldots, u_{d-1})$ if
$$ \langle u_{\nu}, P_{m}P_{n}  \rangle=0, \;\; n \geq md +\nu+1,\,\; m \geq 0,$$
$$  \langle u_{\nu}, P_{m}P_{md+\nu}  \rangle \neq 0, \;\; m \geq 0,$$
for each integer $\nu=0,1,\ldots, d-1$.
It is important to remark that when $d=1$ we meet again the notion of regular orthogonality. Furthermore, the $d$-orthogonality corresponds to the generalisation of the well-known recurrence relation fulfilled by the orthogonal polynomials, as the next Theorem recalls.

\begin{theorem}\cite{Maroni-Toulouse} \label{recurrence relation for d ortho}
Let ${\{P_{n}\}}_{n \geq 0}$ be a MPS. The following assertions are equivalent:
\begin{description}
\item[a)]  ${\{P_{n}\}}_{n \geq 0}$ is $d$-orthogonal with respect to $U=(u_{0},\ldots, u_{d-1})$.
\item[b)] ${\{P_{n}\}}_{n \geq 0}$ satisfies a $(d+1)$-order recurrence relation ($d \geq 1$):
$$P_{m+d+1}(x)=(x-\beta_{m+d})P_{m+d}(x)-\sum_{\nu=0}^{d-1}\gamma_{m+d-\nu}^{d-1-\nu}P_{m+d-1-\nu}(x), \:\: m \geq 0,$$
with initial conditions
$$P_{0}(x)=1,\:\: P_{1}(x)=x-\beta_{0}\;\: \textrm{and if }d \geq 2:$$
$$P_{n}(x)=(x-\beta_{n-1})P_{n-1}(x)-\sum_{\nu=0}^{n-2}\gamma_{n-1-\nu}^{d-1-\nu}P_{n-2-\nu}(x), \:\: 2 \leq n \leq  d,$$
and regularity conditions: $\gamma_{m+1}^{0} \neq 0,\: \; m \geq 0.$
\end{description}
\end{theorem}

The classical orthogonal polynomials fulfil many important properties (e.g. \cite{Chihara}) and they are characterised by the Hahn's property, that is to say, the monic polynomial sequence obtained through the standard differentiation $D$,  $$\{ (n+1)^{-1}DP_{n+1}(x)\}_{n \geq 0}\, ,$$ is also orthogonal. 
The $d$-orthogonal MPSs fulfilling this property are also called classical $d$-orthogonal MPSs in the Hahn's sense.

\begin{definition}\cite{Douak2classiques}
A PS ${\{P_{n}\}}_{n \geq 0}$ is $d$-symmetric if it fulfils
$$P_{n}(\xi_{k}x)=\xi_{k}^{n}P_{n}(x),\;\; n \geq 0, \;\; k=1,2,\ldots, d,$$
where $\xi_{k}=\exp\left( \frac{2ik\pi}{d+1}\right),\;\;
k=1,\ldots,d, \,\; \xi_{k}^{d+1}=1$.

If $d=1$, then $\xi_{1}=-1$ and we meet the definition of a symmetric PS in which we have the following property $P_{n}(-x)=(-1)^{n}P_{n}(x),\;\; n \geq 0.$
\end{definition}

Finally, given $\varpi \in \mathcal{P}$ and $u \in \mathcal{P}'$, the form $\varpi u$, called the
left-multiplication of $u$ by the polynomial $\varpi$, is defined by 
\begin{equation}\label{poly-times-u}
\langle \varpi u,p \rangle= \langle u,\varpi p \rangle,\:\:\: \forall p \in \mathcal{P},
\end{equation}
and the transpose of the derivative operator on $\mathcal{P}$ defined by $p \rightarrow (Dp)(x)=p^{\prime}(x),$ is the following (cf. \cite{theoriealgebrique}):
\begin{equation}\label{funcionalDu}
u \rightarrow Du:\;\;\langle Du,p \rangle=-\langle u,p^{\prime} \rangle ,\:\:\: \forall p \in \mathcal{P},\end{equation}
so that we can retain the usual rule of the derivative of a product when applied to the left-multiplication of a form by a polynomial. Indeed, it is easily established that 
\begin{align}
\label{derivada do produto} &D(pu)=p^{\prime}u+pD(u).
\end{align}

\section{Differential operators on $\mathcal{P}$ and technical identities}\label{operator}

In this section, we summarise the operational point of view proposed in \cite{PM-TAM-2019} that will be followed along the text. The notation here recalled help us in dealing with an operator $J$ that we may consider in the search of polynomial solutions of many different differential identities.

\medskip

Given a sequence of polynomials $\{a_{\nu}(x)\}_{ \nu \geq 0}$, let us consider the following linear mapping $J: \mathcal{P} \rightarrow \mathcal{P}$ (cf. \cite{PM-Korean-2005}, \cite{Pincherle}).
\begin{equation} \label{operatorJ}
J= \sum_{\nu \geq 0} \frac{a_{\nu}(x)}{\nu!} D^{\nu},\quad \deg a_{\nu} \leq \nu,\quad \nu \geq 0.
\end{equation}
Expanding $a_{\nu}(x)$ as follows:
$$a_{\nu}(x)=\sum_{i=0}^{\nu} a_{i}^{[\nu]}x^{i}, $$
and recalling that $ D^{\nu}\left( \xi ^n \right) (x)= \frac{n!}{(n-\nu)!} x^{n-\nu} $, we get the next identities about $J$:
\begin{eqnarray}
\label{Jx^n-short}&& J\left( \xi^n \right) (x) = \sum_{\nu= 0}^{n} a_{\nu} (x) 
\binom{n}{\nu} x^{n-\nu}\\
\label{Jx^n} && J\left( \xi^n \right) (x) = \sum_{\tau= 0}^{n} \left( \sum_{\nu=0}^{\tau}\binom{n}{n-\nu} a_{\tau-\nu}^{[n-\nu]} \right) x^{\tau},\quad n \geq 0.
\end{eqnarray}
Most in particular, a linear mapping $J$ is an isomorphism if and only if 
\begin{equation}\label{iso-conditions}
\deg \left( J\left( \xi^n  \right)(x) \right)= n\;, \;\; n \geq 0, \; \; \textrm{and}\;\; J\left( 1  \right)(x) \neq 0.
\end{equation}

\begin{lemma}\label{lemmaPascal} \cite{PM-TAM-2019}
For any linear mapping $J$, not increasing the degree, there exists a unique sequence of polynomials $\{a_{n}\}_{n \geq 0}$, with $\deg a_{n} \leq n$, so that $J$ is read as in \eqref{operatorJ}. Further, the linear mapping $J$ is an isomorphism of $\mathcal{P}$ if and only if
\begin{equation}\label{lambdan}
 \sum_{\mu=0}^{n}\binom{n}{\mu} a_{\mu}^{[\mu]} \neq 0 , \quad n \geq 0. 
 \end{equation}
\end{lemma}

\medskip
In brief,  an operator that does not increase the degree can always be expressed in a $J$ format. Let us now recall some useful identities regarding any operator $J$, that are obtained by duality and taking into account \eqref{funcionalDu}.
\begin{align*}
\langle ^{t}J(u), f  \rangle &= \langle u, J(f) \rangle \;, \quad u \in \mathcal{P}^{\prime},\quad f \in \mathcal{P},\\
& = \sum_{n \geq 0} \langle u, \frac{a_{n}(x)}{n!}f^{(n)}(x) \rangle = \sum_{n \geq 0} \frac{(-1)^n}{n!} \langle D^{n}\left(a_{n}u\right), f \rangle ;
\end{align*} 
thence
\begin{equation}\label{transpose-J}
^{t}J(u) = \sum_{n \geq 0} \frac{(-1)^n}{n!}  D^{n}\left(a_{n}u\right), \quad u \in \mathcal{P}^{\prime}.
\end{equation}
Henceforth, we will denote the transposed operator $^{t}J(u)$ by $J(u)$, since most of the forthcoming calculations will be done on $\mathcal{P}'$. The results presented in \cite{PM-TAM-2019} about the $J$-image of the product of two polynomials $fg$ and of the form $fu$ require the definition of the following operator $J^{(m)}\, , \; m \geq 0$,  on $\mathcal{P}$ 
\begin{center} $J^{(m)}= \displaystyle \sum_{n \geq 0 } \frac{a_{n+m}(x)}{n!}D^{n}$ \end{center}
 whose transposed operator is given by
\begin{equation}\label{J^(m)}
J^{(m)}(u) = \sum_{n \geq 0 } \frac{(-1)^n}{n!}D^n\left( a_{n+m} u \right), \quad m \geq 0.
\end{equation}

\begin{lemma} \cite{PM-TAM-2019} For any $f,g \in \mathcal{P}$, $u \in \mathcal{P}^{\prime}$, we have:
\begin{align}
\label{J(fg)}&J\left( fg \right)(x)= \sum_{n \geq 0 } J^{(n)}\left(f\right)(x)  \frac{g^{(n)}(x)}{n!} = \sum_{n \geq 0 } J^{(n)}\left(g\right)(x)  \frac{f^{(n)}(x)}{n!},\\
\label{J(fu)}&J\left( fu  \right) = \sum_{n \geq 0 } \frac{(-1)^n}{n!}f^{(n)}(x) J^{(n)}(u).
\end{align}
\end{lemma}

\subsection{Looking at $J$ as a lowering operator of order $k$}

The generic operator $J$ defined in \eqref{operatorJ} does not increase the degree of polynomials and can be categorised in terms of its behaviour when compared with the standard derivative of order $k$, for $k=0,1,2,\ldots$, as we show next.

\medskip

Let us suppose that $J$ is an operator expressed as in \eqref{operatorJ}, and acting as the derivative of order $k$, for some non-negative integer $k$, that is, it fulfils the following conditions.
\begin{align}
\label{deg-k1} & J\left( \xi^{k}\right)(x) = a_{0}^{[k]} \neq 0 \;\; \textrm{and} \; \;\deg\left(J\left( \xi^{n+k}\right)(x)  \right) = n,\quad n \geq 0;  \\
\label{deg-k2}& J\left( \xi^{i}\right)(x) =0,\quad 0 \leq i \leq k-1,\; \textrm{if} \; k\geq 1.
\end{align}

\begin{lemma}\label{J-descending}
An operator $J$ fulfils \eqref{deg-k1}-\eqref{deg-k2} if and only if the next set of conditions hold.
\begin{description}
\item[a)] $a_{0}(x)=\cdots=a_{k-1}(x)=0$, if $k \geq 1$;
\item[b)] $\deg\left( a_{\nu}(x) \right) \leq \nu-k$, $\nu \geq k$;
\item[c)] \begin{equation}\label{neqcondition}
\lambda_{n+k}^{[k]} := \sum_{\nu = 0}^{n} \binom{n+k}{n+k-\nu} a_{n-\nu}^{[n+k-\nu]} \neq 0, \; n \geq 0.
\end{equation}
\end{description}
\end{lemma}

\begin{remark}
Note that in \eqref{neqcondition} we find $\lambda_{k}^{[k]}=a_{0}^{[k]} $.
\newline If $k=0$, then it is assumed that $\lambda_{n}^{[0]} \neq 0,\; n \geq 0$, matching \eqref{lambdan}, so that $J$ is an isomorphism. 
\newline If $k=1$, then $J$ imitates the usual derivative and is commonly called a lowering operator (e.g. \cite{Srivastava-Ben-Cheikh, PM-TAM-2016}).
\end{remark}
Given a MPS $\{ P_{n} \}_{n \geq 0}$ and a non-negative integer $k$, let us define its (normalised) $J$-image sequence of polynomials as follows, and notate its dual sequence by $\{ \tilde{u}_{n}\}_{n \geq 0}$. Of course, given that $J$ satisfies \eqref{deg-k1}-\eqref{deg-k2}, the obtained polynomial sequence $\{ \tilde{P}_{n} \}_{n \geq 0}$ is also a MPS.  
\begin{equation}\label{JPn}
\tilde{P}_{n}(x)= \left( \lambda_{n+k}^{[k]} \right)^{-1}J\left( P_{n+k} (x) \right), \quad n \geq 0.
\end{equation}
Moreover, the dual sequences of $\{ P_{n} \}_{n \geq 0}$ and $\{ \tilde{P}_{n} \}_{n \geq 0}$ are related as the next Lemma points out.
\begin{lemma} \cite{PM-TAM-2019} \label{J(tilde un)} Let us consider a MPS $\{ P_{n} \}_{n \geq 0}$ and an operator $J$ of the form \eqref{operatorJ} fulfilling \eqref{deg-k1}-\eqref{deg-k2}. Thus,
\par \begin{center} $\quad J\left( \tilde{u}_{n} \right) = \lambda_{n+k}^{[k]} u_{n+k}$. \end{center}
\end{lemma}

\begin{lemma} \cite{PM-TAM-2019} \label{tildeP=P} Let us consider a MPS $\{ P_{n} \}_{n \geq 0}$ and an operator $J$ of the form \eqref{operatorJ} such that \eqref{deg-k1}-\eqref{deg-k2} hold. Thus,
\begin{center} $\tilde{P}_{n}(x) = P_{n}(x),\;\; n \geq 0,\quad$ if and only if $\quad J\left( u_{n} \right) = \lambda_{n+k}^{[k]} u_{n+k}$. \end{center}
\end{lemma}

In brief, given a non-negative integer $k$, we can address the general problem of finding the MPSs $\textbf{P} = \left( P_{0}, P_{1}, \ldots \right)^{t}$ such that 
\begin{equation}
 \label{general-matrix-identity}\textbf{P} = \Lambda^{[k]}_{J} \; J \left( \textbf{P} \right),
 \end{equation}
where the matrix (with infinite dimensions)  $\Lambda^{[k]}_{J}$ contains de normalisation constants $ \left(\lambda_{n+k}^{[k]}\right)^{-1}\, , \, n \geq 0$, by working with the dual sequence
$\textbf{u}= \left( u_{0}, u_{1}, \ldots \right)^{t}$.

\section{An isomorphism applied to a $2$-orthogonal sequence }\label{sec:k=0 and J third order} 

Throughout this paper, we consider that $J$ is an isomorphism and $a_{\nu}(x)=0\,,\; \nu \geq 4$, thus
\begin{align}
\label{third-order-J}& J = a_{0}(x)I + a_{1}(x)D+\frac{a_{2}(x)}{2}D^{2}+\frac{a_{3}(x)}{3!}D^{3}\, ,\, \textrm{where}\\
\nonumber & a_{0}(x)=a_{0}^{[0]}\; , \; a_{1}(x)=a_{0}^{[1]}+a_{1}^{[1]}x \; , \; a_{2}(x)=a_{0}^{[2]}+a_{1}^{[2]}x +a_{2}^{[2]}x^2 \; , \\
\nonumber & a_{3}(x)=a_{0}^{[3]}+a_{1}^{[3]}x +a_{2}^{[3]}x^2   + a_{3}^{[3]}x^3   \;, \; 
\end{align}
and we suppose that the MPS $\{P_{n} \}_{n \geq 0}$ fulfils $$J\left( P_{n}(x) \right)= \lambda^{[0]}_{n}P_{n}(x), \; \textrm{with} \; \lambda^{[0]}_{n} \neq 0\, ,n \geq 0\, .$$

\medskip

Focusing on the $2$-orthogonality, where a polynomial sequence fulfil specific orthogonal conditions towards the two initial elements of the dual sequence \cite{Maroni-Toulouse}, 
we begin by recalling that a  $2$-orthogonal MPS  can be recursively computed by means of three sequences of constants 
$\{\beta_{n}\}_{n \geq 0}$, $\{\gamma^{1}_{n}\}_{n \geq 1}$ and $\{\gamma^{0}_{n}\}_{n \geq 1}$, with $\gamma^{0}_{n+1}  \neq 0,\; n \geq 0$ \cite{Douak-two-Laguerre, Maroni-NumAlg}. For the sake of simplicity, it  has been used (e.g. \cite{Douak-two-Laguerre}), the following notation for these structure coefficients of the $2$-orthogonal polynomial sequences:  $\{\beta_{n}\}_{n \geq 0}$, $\{\alpha_{n}\}_{n \geq 1}$ and $\{\gamma_{n}\}_{n \geq 1}$, respectively, as indicated below in \eqref{ic-2orto}-\eqref{rr-2orto}, with $\gamma_{n+1}  \neq 0,\; n \geq 0$. 
\begin{align}
\label{ic-2orto}  & P_{0}(x)=1,\:\: P_{1}(x)=x-\beta_{0},\;\:  P_{2}(x)=\left( x-\beta_{1} \right) P_{1}(x) - \alpha_{1}; \\
\label{rr-2orto}  & P_{n+3}(x) = \left( x-\beta_{n+2} \right) P_{n+2}(x) - \alpha_{n+2} P_{n+1}(x) - \gamma_{n+1} P_{n}(x).
\end{align}
In a similar manner to the orthogonal case, its dual sequence fulfils a recurrence relation based on those structure coefficients, as follows \cite{Maroni-NumAlg}.
\begin{equation} \label{functional-2orto}
x u_{n} = u_{n-1} + \beta_{n}u_{n} + \alpha_{n+1} u_{n+1} +  \gamma_{n+1} u_{n+2}, \;  \textrm{with} \; n \geq 0, \,\; u_{-1} = 0.
\end{equation}

Moreover, all elements of the dual sequence can be written in terms of the regular functional vector $(u_{0}, u_{1})$. In particular, we have \cite{Maroni-NumAlg} (p. 307)
\begin{align}
\label{Eq-u-even}u_{2n} &= E_{n}(x) u_{0} + A_{n-1}(x) u_{1}\\
\label{Eq-u-odd} u_{2n+1} &= B_{n}(x) u_{0} + F_{n}(x) u_{1}\; 
\end{align}
where $\deg\left( E_{n}(x)  \right)= \deg\left( F_{n}(x)  \right)=n$, $ \deg\left( A_{n}(x)  \right) \leq n $, $\deg\left( B_{n}(x)  \right) \leq n$, 
\newline $E_{0}(x)=1\, ,\, A_{-1}(x)=0$, $B_{0}(x)=0$, $F_{0}(x)=1$. These polynomial coefficients fulfil the following recurrence relations \cite{Maroni-NumAlg}. 
\begin{align*}
& E_{1}(x)= \frac{1}{\gamma_{1}} \left( x-\beta_{0}\right); \quad \quad A_{0}(x)= - \frac{\alpha_{1}}{\gamma_{1}}; \\
&  \alpha_{2n+2}E_{n+1}(x)+E_{n}(x)=\left( x-\beta_{2n+1}\right)B_{n}(x)-\gamma_{2n+2}B_{n+1}(x);\\
&  \gamma_{2n+3}E_{n+2}(x)- \left( x-\beta_{2n+2}\right) E_{n+1}(x)=- B_{n}(x)-  \alpha_{2n+3}B_{n+1}(x);\\
&  \\
& \gamma_{2n+2}F_{n+1}(x)-\left( x-\beta_{2n+1}\right) F_{n}(x)= - A_{n-1}(x)  - \alpha_{2n+2}A_{n}(x);\\
& \alpha_{2n+3}F_{n+1}(x) + F_{n}(x)= \left( x-\beta_{2n+2}\right) A_{n}(x)-  \gamma_{2n+3}A_{n+1}(x)\, , \quad n \geq 0.\\
\end{align*}

Recalling the definition \eqref{neqcondition}, we have
$$\lambda_{n}^{[0]} := \sum_{\nu = 0}^{n} \binom{n}{n-\nu} a_{n-\nu}^{[n-\nu]} \neq 0, \; n \geq 0\, ;$$
so that,
$\lambda_{0}^{[0]}=a_{0}(x)=a_{0}^{[0]}$ ; $ \lambda_{1}^{[0]}= a_{0}^{[0]}+a_{1}^{[1]}$ ;
$\lambda_{2}^{[0]}=a_{0}^{[0]}+2 a_{1}^{[1]} + a_{2}^{[2]}$,  just to mention a few.

In an equivalent manner, we learn from Lemma \ref{tildeP=P} that $J\left( u_{n} \right) = \lambda_{n}^{[0]} u_{n}$, which means
\begin{equation}\label{J(u_n)}
D\left(  -a_{1}(x)u_{n} + \frac{1}{2!}D\left( a_{2}(x)u_{n}\right)-\frac{1}{3!}D^{2}\left(a_{3}(x)u_{n}\right)  \right)= \left( \lambda_{n}^{[0]}-a_{0}(x) \right)u_{n}\, , n \geq 0\, .
\end{equation}
Taking into account the definition of $J$ established in \eqref{third-order-J}, identity \eqref{J(fu)}  allow us to expand the image of the form $fu$, with $f \in \mathcal{P}$ and $u \in \mathcal{P}'$, as follows:
\begin{equation}\label{J(fu)-specific}
J\left(fu \right)= f(x)J(u)-f^{\prime}(x)J^{(1)}(u)+ \frac{1}{2!} f^{\prime\prime}(x)J^{(2)}(u) -\frac{1}{3!} f^{(3)}(x)J^{(3)}(u).
\end{equation}
Furthermore, the forms $J^{(1)}(u)$, $J^{(2)}(u)$ and $J^{(3)}(u)$ have the following definitions that can be deduce from identity \eqref{J^(m)} , for an operator $J$ of third order \eqref{third-order-J}:
\begin{align}
\label{J^{(1)}} & J^{(1)}(u)=a_{1}(x)u-D\left(a_{2}(x) u \right)+\frac{1}{2!}D^{2}\left( a_{3}(x) u \right)  \\
&  J^{(2)}(u)=a_{2}(x)u-D\left(a_{3}(x) u \right)\\
& J^{(3)}(u)=a_{3}(x)u\\
& J^{(m)}(u)=0\, , \; m \geq 4.
\end{align}
When we apply $J$ on \eqref{Eq-u-even} and \eqref{Eq-u-odd} taking into account \eqref{J(fu)-specific}, we get for $n \geq 0$, respectively:
\begin{align}
\label{J(u2n)}  &\lambda_{2n}^{[0]} u_{2n}= E_{n}(x)J(u_{0}) -E'_{n}(x)J^{(1)}(u_{0}) +\frac{1}{2!}E^{(2)}_{n}(x)J^{(2)}(u_{0}) \\
\nonumber & \quad \quad -\frac{1}{3!}E^{(3)}_{n}(x)J^{(3)}(u_{0}) \\
 \nonumber & +A_{n-1}(x)J(u_{1}) -A'_{n-1}(x)J^{(1)}(u_{1}) +\frac{1}{2!}A^{(2)}_{n-1}(x)J^{(2)}(u_{1}) \\
\nonumber & \quad \quad -\frac{1}{3!}A^{(3)}_{n-1}(x)J^{(3)}(u_{1}) \, ;\\
\label{J(u2n+1)}   &\lambda_{2n+1}^{[0]} u_{2n+1}= B_{n}(x)J(u_{0}) -B'_{n}(x)J^{(1)}(u_{0}) +\frac{1}{2!}B^{(2)}_{n}(x)J^{(2)}(u_{0})\\
\nonumber & -\frac{1}{3!}B^{(3)}_{n}(x)J^{(3)}(u_{0}) \\
\nonumber & +F_{n}(x)J(u_{1}) -F'_{n}(x)J^{(1)}(u_{1}) +\frac{1}{2!}F^{(2)}_{n}(x)J^{(2)}(u_{1})\\
\nonumber & -\frac{1}{3!}F^{(3)}_{n}(x)J^{(3)}(u_{1}) \, .
\end{align}
Considering these two latest functional identities with $n=1$ and $n=2$ we put in evidence the following set of relations, while taking $n=0$ only yields trivial identities.
 \begin{align}
\label{7.1} & \lambda_{2}^{[0]} u_{2}= E_{1}(x)\lambda_{0}^{[0]} u_{0} - \frac{1}{\gamma_{1}}J^{(1)}(u_{0}) + A_{0}(x)\lambda_{1}^{[0]} u_{1}\, ; \\
\label{7.2} & \lambda_{4}^{[0]} u_{4}= E_{2}(x)\lambda_{0}^{[0]} u_{0} - E'_{2}(x)J^{(1)}(u_{0})+ \frac{1}{2}E^{(2)}_{2}(x)J^{(2)}(u_{0})\\
\nonumber & \quad \quad \quad \; + A_{1}(x) \lambda_{1}^{[0]} u_{1}- A'_{1}(x)J^{(1)}(u_{1})\, ;
\end{align}
 \begin{align}
\label{8.1} &\lambda_{3}^{[0]} u_{3}= B_{1}(x)\lambda_{0}^{[0]} u_{0} - B'_{1}(x)J^{(1)}(u_{0}) + F_{1}(x)\lambda_{1}^{[0]} u_{1} - F'_{1}(x)J^{(1)}(u_{1})\, ; \\
\label{8.2} &\lambda_{5}^{[0]} u_{5}= B_{2}(x)\lambda_{0}^{[0]} u_{0} - B'_{2}(x)J^{(1)}(u_{0})+ \frac{1}{2}B^{(2)}_{2}(x)J^{(2)}(u_{0})\\
\nonumber &\quad \quad \quad \; + F_{2}(x)\lambda_{1}^{[0]} u_{1} - F'_{2}(x)J^{(1)}(u_{1})+ \frac{1}{2}F^{(2)}_{2}(x)J^{(2)}(u_{1})\, .
\end{align}
In view of \eqref{Eq-u-even}-\eqref{Eq-u-odd}, we highlight the next functional identities
\begin{align}
\label{Eq-u2} u_{2} &= E_{1}(x) u_{0} + A_{0}(x) u_{1}\\
\label{Eq-u3} u_{3} &= B_{1}(x) u_{0} + F_{1}(x) u_{1}\\
\label{Eq-u4} u_{4} &= E_{2}(x) u_{0} + A_{1}(x) u_{1}\\
\label{Eq-u5} u_{5} &= B_{2}(x) u_{0} + F_{2}(x) u_{1}\
\end{align}
where
\begin{align*}
& E_{1}(x)= \frac{1}{\gamma_{1}} \left( x-\beta_{0}\right); \quad A_{0}(x)= - \frac{\alpha_{1}}{\gamma_{1}}; \quad 
B_{1}(x) = - \frac{\alpha_{2}}{\gamma_{1}\gamma_{2}} \left( x-\beta_{0}\right) - \frac{1}{\gamma_{2}} ; \\
& F_{1}(x) = \frac{1}{\gamma_{2}} \left( x-\beta_{1} + \frac{\alpha_{2}\alpha_{1}}{\gamma_{1}} \right); \quad E_{2}(x) = \frac{1}{\gamma_{3}} \big( \left( x-\beta_{2}\right) E_{1}(x) -\alpha_{3} B_{1}(x) \big); \\
& A_{1}(x) = -\frac{1}{\gamma_{3}} \left( \alpha_{3}F_{1}(x)+1+\frac{\alpha_{1}}{\gamma_{1}}\left( x-\beta_{2} \right)\right);\\
& B_{2}(x) = \frac{1}{\gamma_{4}}\Big((x-\beta_{3})B_{1}(x)-\alpha_{4}E_{2}(x)-E_{1}(x)\Big); \\
& F_{2}(x) =\frac{1}{\gamma_{4}}\Big((x-\beta_{3})F_{1}(x)-\alpha_{4}A_{1}(x)-A_{0}(x)\Big).
\end{align*}

From \eqref{7.1} and \eqref{8.1}, together with \eqref{Eq-u2}-\eqref{Eq-u3}, we get respectively
\begin{align}\label{9.1}
&J^{(1)}(u_{0})=p_{0}(x)u_{0}+p_{1}(x)u_{1}\; \, , \textrm{ with}\\
\nonumber &  p_{0}(x)=\gamma_{1} E_{1}(x)\left( \lambda_{0}^{[0]}-\lambda_{2}^{[0]} \right)\, , \quad p_{1}(x)=\gamma_{1}A_{0}(x) \left( \lambda_{1}^{[0]}-\lambda_{2}^{[0]}  \right) \,; 
\end{align}
\begin{align}\label{9.2}
&J^{(1)}(u_{1})=f_{0}(x)u_{0}+f_{1}(x)u_{1}\;\, \textrm{with}\\
\nonumber & f_{0}(x)=\gamma_{2} \left( \lambda_{0}^{[0]}-\lambda_{3}^{[0]} \right)B_{1}(x) +\alpha_{2}E_{1}(x)\left( \lambda_{0}^{[0]}-\lambda_{2}^{[0]} \right)\, , \quad \\
\nonumber & f_{1}(x)=\gamma_{2}\left( \lambda_{1}^{[0]}-\lambda_{3}^{[0]} \right)F_{1}(x) + \alpha_{2}A_{0}(x)\left( \lambda_{1}^{[0]}-\lambda_{2}^{[0]}  \right) u_{1}\,. 
\end{align}

Let us  now focus on \eqref{7.2} and \eqref{Eq-u4}, having learned the above expansions \eqref{9.1} and \eqref{9.2}.
The polynomial $E_{2}(x)$ has degree two and $E^{(2)}_{2}(x)=\frac{2}{\gamma_1 \gamma_3} \neq 0$, thus we may isolate in \eqref{7.2}  the term $J^{(2)}(u_0)$, obtaining:
\begin{align}\label{9.3}
&J^{(2)}(u_0) = \overline{p}_{0}(x)u_{0}+\overline{p}_{1}(x)u_{1}\; \, , \textrm{ with}\\
\nonumber &  \overline{p}_{0}(x)=\gamma_{1}\gamma_3\Big(   \left( \lambda_{4}^{[0]}-\lambda_{0}^{[0]}  \right)E_{2}(x)+ E'_{2}(x)p_{0}(x)+A'_{1}(x)f_{0}(x) \Big)\, , \\
\nonumber & \quad \overline{p}_{1}(x)=\gamma_{1}\gamma_3 \Big(   \left( \lambda_{4}^{[0]}-\lambda_{1}^{[0]}  \right)A_{1}(x)+ E'_{2}(x)p_{1}(x)+A'_{1}(x)f_{1}(x) \Big) \,. 
\end{align}
We remark that $\deg\left(\overline{p}_{0}(x)\right) \leq 2$ and $\deg\left(\overline{p}_{1}(x)\right) \leq 1$.
\par Analogously, identity  \eqref{8.2} allows us to write $J^{(2)}(u_1)$ in terms of the pair $(u_0,u_1)$, with the help of the information settled up to this point, in particular  \eqref{Eq-u5}.
More precisely, the polynomial $F_{2}(x)$ has degree two and $F^{(2)}_{2}(x)=\frac{2}{\gamma_2 \gamma_4} \neq 0$, and we may affirm that
\begin{align}\label{9.4}
&J^{(2)}(u_1) = \overline{f}_{0}(x)u_{0}+\overline{f}_{1}(x)u_{1}\; \, , \textrm{ with}\\
\nonumber &  \overline{f}_{0}(x)=\gamma_{2}\gamma_{4}\Big(   \left( \lambda_{5}^{[0]}-\lambda_{0}^{[0]}  \right)B_{2}(x)+ B'_{2}(x)p_{0}(x)+F'_{2}(x)f_{0}(x)-\frac{1}{2}B^{(2)}_{2}(x)\overline{p}_{0}(x) \Big)\, , \\
\nonumber & \overline{f}_{1}(x)=\gamma_{2}\gamma_{4}\Big(   \left( \lambda_{5}^{[0]}-\lambda_{1}^{[0]}  \right)F_{2}(x)+ B'_{2}(x)p_{1}(x)+F'_{2}(x)f_{1}(x)-\frac{1}{2}B^{(2)}_{2}(x)\overline{p}_{1}(x) \Big)\,. 
\end{align}
We remark that $\deg\left(\overline{f}_{0}(x)\right) \leq 2$ and $\deg\left(\overline{f}_{1}(x)\right) \leq 2$.

\smallskip

Next, we list a small set of functional identities that are fulfilled by the fundamental pair of functionals $(u_0,u_1)$ when the identity $J\left( P_{n} (x) \right) =  \lambda_{n}^{[0]} P_{n}(x)$ holds.
\begin{lemma}\label{basic-lemma}
Considering an isomorphism $J$ defined by \eqref{third-order-J} and a $2$-orthogonal MPS $\{P_{n}(x)\}_{n \geq 0}$ such that $J\left( P_{n} (x) \right) =  \lambda_{n}^{[0]} P_{n}(x)\, , \, n \geq 0$ , the initial elements of the corresponding dual sequence $\{u_{n} \}_{n \geq 0}$ fulfil the following three identities:
\begin{align}
\label{Da2u0} & D\left( a_{2}(x)u_{0}\right)= \left( 2p_{0}(x)+ 4a_{1}(x) \right)  u_{0}+2p_{1}(x)u_{1}, \\
\label{Da2u1} & \frac{1}{2}D^2\left( a_{2}(x)u_{1}\right)-3a_{1}^{[1]}u_{1} = D \Big( f_{0}(x)u_{0}+ \left(2a_{1}(x) +f_{1}(x) \right)u_{1} \Big), \\
\label{Dcomplete} &D \Big( \overline{p}_{0}(x)u_{0}+\overline{p}_{1}(x)u_{1} \Big)+ \left(  2a_{1}(x)+4p_{0}(x) \right) u_{0} + 4p_{1}(x) u_{1} = 0 \,,
\end{align}
where polynomials $p_{0}(x)$ and $p_{1}(x)$ are defined in \eqref{9.1}, the polynomials $f_{0}(x)$ and $f_{1}(x)$ are defined in \eqref{9.2}, and the polynomials  $\overline{p}_{0}(x)$ and $\overline{p}_{1}(x)$ are defined in \eqref{9.3}.
\end{lemma}
\begin{proof}
Taking $n=0$ in \eqref{J(u_n)} , we get
$$D\left( -a_{1}(x)u_{0} + \frac{1}{2!}D\left( a_{2}(x)u_{0}\right)-\frac{1}{3!}D^{2}\left(a_{3}(x)u_{0}\right)     \right)=0$$
which implies $-a_{1}(x)u_{0} + \frac{1}{2!}D\left( a_{2}(x)u_{0}\right)-\frac{1}{3!}D^{2}\left(a_{3}(x)u_{0}\right) =0 \, ,$
that is, 
\begin{equation}\label{eqD2a3u0}
D^{2}\left(a_{3}(x)u_{0}\right) = -6a_{1}(x)u_{0} +3D\left( a_{2}(x)u_{0}\right)\,.
\end{equation}

On the other hand, identity \eqref{J^{(1)}} asserts
$$J^{(1)}(u_0) = a_{1}(x)u_{0}-D\left( a_{2}(x)u_{0} \right)  + \frac{1}{2!}D^{2}\left( a_{3}(x)u_{0}\right) $$
by which we get the following identity replacing $D^{2}\left(a_{3}(x)u_{0}\right)$ by the above expression \eqref{eqD2a3u0} and $J^{(1)}(u_{0})$ by $p_{0}(x)u_{0}+p_{1}(x)u_{1}$:
$$ D\left( a_{2}(x)u_{0}\right)= \left( 2p_{0}(x)+ 4a_{1}(x) \right)  u_{0}+2p_{1}(x)u_{1}.$$

The second identity is deduced in a similar way. We begin by considering $n=1$ in \eqref{J(u_n)}  with $D^2\left( a_{3}(x)u_{1} \right)$ replaced by $2J^{(1)}(u_{1})-2a_{1}(x)u_{1}+2D\left(a_{2}(x)u_{1} \right)$:
$$D\Big( -\frac{4}{6}a_{1}(x)u_{1}+\frac{1}{6}D\left(a_{2}(x)u_{1} \right)-\frac{1}{3}J^{(1)}(u_{1})   \Big)= \left(  \lambda_{1}^{[0]} -a_{0}(x) \right)u_{1}.$$
When we substitute $J^{(1)}(u_{1}) $ by the expression provided by \eqref{9.2} and calculate
$\lambda_{1}^{[0]} -a_{0}(x)=a_{1}^{[1]}$ we get \eqref{Da2u1}.
With respect to the third relation, taking into account \eqref{9.1} and \eqref{9.3}, and in view of 
\begin{align*}
J^{(1)}(u)=a_{1}(x)u-D\left(a_{2}(x) u \right)+\frac{1}{2!}D^{2}\left( a_{3}(x) u \right)  \\
J^{(2)}(u)=a_{2}(x)u-D\left(a_{3}(x) u \right)
\end{align*}
we may write
\begin{align*}
a_{1}(x)u_{0}-D\left(a_{2}(x) u_{0} \right)+\frac{1}{2!}D^{2}\left( a_{3}(x) u_{0} \right) = p_{0}(x)u_{0}+p_{1}(x)u_{1}\, ,  \\
a_{2}(x)u_{0}-D\left(a_{3}(x) u_{0} \right)= \overline{p}_{0}(x)u_{0}+\overline{p}_{1}(x)u_{1} \,.
\end{align*}
Let us apply the operator $D$ on the second identity in order to eliminate $D^{2}\left( a_{3}(x) u_{0} \right)$ in the first one, yielding
\newline $D\left(a_{2}(x) u_{0} \right)+ D\Big(\overline{p}_{0}(x)u_{0}+\overline{p}_{1}(x)u_{1} \Big)+ \left( 2p_{0}(x) -2a_{1}(x) \right) u_{0} + 2p_{1}(x) u_{1}  = 0\, .$ 
Inserting the information of \eqref{Da2u0}, we end the proof. \end{proof}

\section{First choice in the polynomial coefficients of operator $J$}\label{sec:a2=0} 

We will now proceed by taking two assumptions regarding polynomial coefficients $a_{1}(x)$ and $a_{2}(x)$ of operator $J$. These hypotheses were suggested by different calculations, performed during this research, that aimed to achieve functional identities corresponding to the Hahn classical set up, as described in p.182-183 of \cite{d-classical}.

In this section, we require  the use of Lemma 1.2, page 297 of  \cite{Maroni-Semi-classical} that has the following content.

\begin{lemma} \label{lemma1.2} \cite{Maroni-Semi-classical} 
Let $M$ and $N$ be two polynomials such that $Mu_{0}=Nu_{1}$.
If the vector functional $U=(u_{0}, u_{1})^{T}$ is regular, then necessarily $M=0$ and $N=0$.
\end{lemma}

Let us analyse the identities \eqref{Da2u0} , \eqref{Da2u1}  and \eqref{Dcomplete} assuming that
\begin{align}\label{assumptions}
a_{2}(x) = 0 \; \textrm{and} \; a_{1}(x)=-\frac{1}{3}E_{1}(x)=-\frac{1}{3\gamma_1}(x-\beta_0)\,.
\end{align}

Under  \eqref{assumptions} , equation \eqref{Da2u0} yields $\left( 2p_{0}(x)+ 4a_{1}(x) \right)  u_{0}+2p_{1}(x)u_{1}=0$. Taking into account Lemma \ref{lemma1.2} we get
\begin{align}
&\label{p0} p_{0}(x)=-2a_{1}(x)\\
&\label{p1=0} p_{1}(x)=0  \; ,\, \textrm{thus}\, , \; \alpha_{1}=0 \, .
\end{align}

Looking at  \eqref{Dcomplete} with these new informations about $p_{0}(x)$ and $p_{1}(x)$ we get the following:
\begin{align}
\nonumber & D \Big( \overline{p}_{0}(x)u_{0}+\overline{p}_{1}(x)u_{1} \Big)-6a_{1}(x)u_{0}  = 0\\
\label{EqClassic-2} & \textrm{or} \;  D \Big( \overline{p}_{0}(x)u_{0}+\overline{p}_{1}(x)u_{1} \Big)+2E_{1}(x)u_{0} + 2A_{0}(x) u_{1} = 0 \\
\nonumber & \textrm{with} \; A_{0}(x)=0.
\end{align}

Finally, from \eqref{Da2u1} we obtain:
\begin{align} \label{EqClassic-1}
 D \Big( -\gamma_{1}f_{0}(x)u_{0} - \gamma_{1}\left(2a_{1}(x) +f_{1}(x) \right)u_{1} \Big) +u_{1} = 0\, .
\end{align}

In brief, the two identities \eqref{EqClassic-1} and \eqref{EqClassic-2} allow us to affirm that the regular vector functional $U=(u_{0}, u_{1})^{T}$ fulfil
$D\left(\Phi(x) U \right) + \Psi(x) U = 0$ with
\medskip

\begin{tabular}{cc}
$\quad\quad$ $\Psi(x)=\left[\begin{array}{cc}0 & 1 \\ 2E_{1}(x) & 2A_{0}(x) \end{array}\right], $
 & 
$\Phi(x)= \left[\begin{array}{cc} \phi_{1,1}(x) & \phi_{1,2}(x) \\  \phi_{2,1}(x) & \phi_{2,2}(x)\end{array}\right]$,
 \end{tabular}
 \medskip
\noindent as it is established for the $2$-orthogonal polynomial sequences fulfilling Hahn's property in  p.182-183 of \cite{d-classical} (see also p.320 of \cite{Maroni-Semi-classical}).

In particular, we know that
$\deg \left( \phi_{1,1}(x)\right) \leq 1$, $\deg \left( \phi_{1,2}(x)\right) \leq 1$, $\deg \left( \phi_{2,1}(x)\right) \leq 2$ and $\deg \left( \phi_{2,2}(x)\right) \leq 1$  as required in \cite{d-classical}, as explained next with the detailed description of the functional relations and the polynomials involved.
\begin{align}
\nonumber & D\left( \phi_{1,1}(x) u_{0} + \phi_{1,2}(x) u_{1}\right) +u_{1} = 0,\\
\nonumber & D\left( \phi_{2,1}(x) u_{0} + \phi_{2,2}(x) u_{1}\right) +2E_{1}(x)u_{0}+2A_{0}(x)u_{1} = 0,\;\; \textrm{with}\; A_{0}(x)=0,
\end{align}
\begin{align}
\label{phi-1,1} &\quad \phi_{1,1}(x) =-\gamma_{1}f_{0}(x), \textrm{that is,}\;  \\
\nonumber & \phi_{1,1}(x)  = \alpha_2 \left( \frac{1}{3\gamma_{1}}-a_{3}^{[3]} \right) x + a_{3}^{[3]}\left( \alpha_{2}\beta_{0}-\gamma_{1}\right)- \frac{\alpha_{2}\beta_{0}}{3\gamma_{1}} + 1\\
\label{phi-1,2} &\quad \phi_{1,2}(x) =- \gamma_{1}\left(2a_{1}(x) +f_{1}(x) \right), \textrm{that is,}\; \\
\nonumber & \phi_{1,2}(x) = a_{3}^{[3]} \gamma_{1} x + \frac{1}{3} \left( -2\beta_{0}+ \beta_{1}(2-3a_{3}^{[3]}\gamma_{1} )\right) 
\end{align}
\begin{align}
\label{phi-2,1} &\quad \phi_{2,1}(x) = \overline{p}_{0}(x), \textrm{that is,}\; \\
\nonumber & \phi_{2,1}(x) =4 a_{3}^{[3]} x^2 \\
\nonumber & + \frac{1}{3\gamma_{1}\gamma_{2}} \Big(
\alpha_{2}\alpha_{3}\left( -1+9a_{3}^{[3]}\gamma_{1} \right)-2\left(\beta_{2}(-1+6a_{3}^{[3]} \gamma_{1} ) + \beta_{0}(1+6a_{3}^{[3]} \gamma_{1} ) \right)\gamma_{2} \Big) x\\
\nonumber &+  \frac{1}{3\gamma_{1}\gamma_{2}} \Big( \alpha_{3}\left( \alpha_{2}\beta_{0}-\gamma_{1} \right)\left(1-9a_{3}^{[3]}\gamma_{1}\right) +2\beta_{0}\left(\beta_{0}+\beta_{2}\left(-1+6a_{3}^{[3]}\gamma_{1}\right)   \right)\gamma_{2} \Big)\\
\label{phi-2,2} &\phi_{2,2}(x) =  \overline{p}_{1}(x), \textrm{that is,}\;\\
\nonumber &\phi_{2,2}(x) = \frac{1}{3\gamma_{2}} \alpha_{3}\left(1-9a_{3}^{[3]}\gamma_{1}\right)   x +  \frac{1}{3\gamma_{2}} \left( \alpha_{3}\beta_{1} \left(-1+9a_{3}^{[3]}\gamma_{1}\right)+ 3 \left(1-4a_{3}^{[3]}\gamma_{1}\right) \gamma_2 \right).
\end{align}

Moreover, the characterisation of the classical $d$-orthogonal polynomial sequences (in Hahn's sense) provided by \cite{d-classical}  imposes that the coefficient of $x^2$ in $\phi_{2,1}(x)$ is different from $\frac{2}{\gamma_{1}(m+1)}$, $m \geq 0$, (that is, different from $c \times \frac{1}{m+1}$, being $c$ the leading coefficient of $\psi(x)=2E_{1}(x)$) and the coefficient of $x$ in $\phi_{1,2}(x)$ is different from $\frac{1}{m+1}$, $m \geq 0$ (cf. p.183 of \cite{d-classical} or p.320 of \cite{Maroni-Semi-classical}). 
This conducts us to the addition of the following restriction about the leading coefficient of polynomial $a_{3}(x)$ in the definition of the operator $J$:
$$a_{3}^{[3]} \neq \frac{1}{\gamma_{1}(m+1)}\; , m\geq 0. $$
We now summarise the conclusions of the above argumentation in the next result.

\begin{theorem}\label{a2=0}
 Let  $\{P_{n}(x)\}_{n \geq 0}$ be a $2$-orthogonal MPS such that $J\left( P_{n} (x) \right)= \lambda_{n}^{[0]}P_{n}(x), \; n \geq 0$, with $J$ defined by \eqref{third-order-J}.
 If the polynomial coefficients of operator $J$ fulfil:
\begin{align*}
a_{2}(x) = 0 \; \wedge a_{1}(x)=-\frac{1}{3\gamma_{1}}\left( x-\beta_{0} \right) \; \wedge \, a_{3}^{[3]} \neq \frac{1}{\gamma_{1}(m+1)}\; , m\geq 0, 
\end{align*}
then $\{P_{n}(x)\}_{n \geq 0}$ is (Hahn) classical and the regular vector functional $U=(u_{0}, u_{1})^{T}$ fulfils
$$D\left(\Phi(x) U \right) + \Psi(x) U = 0,$$

\begin{tabular}{cc}
$\quad\quad$ $\Psi(x)=\left[\begin{array}{cc}0 & 1 \\ 2E_{1}(x) & 0 \end{array}\right], $
 & 
$\Phi(x)= \left[\begin{array}{cc} \phi_{1,1}(x) & \phi_{1,2}(x) \\  \phi_{2,1}(x) & \phi_{2,2}(x)\end{array}\right]$,
 \end{tabular}

\noindent with entries $\phi_{i,j}(x)$ defined by \eqref{phi-1,1}-\eqref{phi-2,2}. 
\end{theorem}

\section{Second choice in the polynomial coefficients of operator $J$}\label{sec:a3=varsigma a2} 

In this section, we consider the following  hypotheses regarding the polynomial coefficients of operator $J$, defined in \eqref{third-order-J}:
$$a_{3}(x)=\tau\, a_{2}(x)\, , \textrm{for a certain nonzero constant\;} \tau \, , \textrm{and}\;  \deg(a_{1}(x))=1.$$

From the definition of $J^{(2)}(u)$ and the content of \eqref{9.4} , we may write
$$ a_{2}(x)u_1-D\left(a_{3}(x)u_1\right) = \overline{f}_{0}(x)u_{0}+\overline{f}_{1}(x)u_{1}$$
which in view of $a_{3}(x)=\tau\, a_{2}(x)$ becomes
$$ \tau D\left(a_{2}(x)u_1\right) = -\overline{f}_{0}(x)u_{0}+\left( a_{2}(x)-\overline{f}_{1}(x)\right)u_{1}$$
and thus
\begin{equation}\label{D^2(a2)}
 D^2\left(a_{2}(x)u_1\right) = \frac{1}{\tau}D\left( -\overline{f}_{0}(x)u_{0}+\left( a_{2}(x)-\overline{f}_{1}(x)\right)u_{1}\right).
 \end{equation}
We can now eliminate the term on $D^2\left(a_{2}(x)u_1\right) $ of the general identity \eqref{Da2u1} of Lemma \ref{basic-lemma}; in other words, reading \eqref{Da2u1} with the information of \eqref{D^2(a2)} yields
\begin{align}
&\label{Eq-case2-1} D\Big( \varpi_{1,1}(x)u_0+ \varpi_{1,2}(x)u_1  \Big) +u_1 = 0 \;, \textrm{with}\, \\
&\label{varpi11} \varpi_{1,1}(x)= \frac{1}{3a_{1}^{[1]}} \left( \frac{1}{2\tau}\overline{f}_{0}(x) + f_{0}(x)    \right)\\
&\label{varpi12} \varpi_{1,2}(x)= \frac{1}{3a_{1}^{[1]}} \left( 2a_{1}(x)+f_{1}(x)-\frac{1}{2\tau}\left(a_{2}(x)- \overline{f}_{1}(x) \right)\right)\,.
\end{align}
Looking at the definition of the polynomials $\overline{f}_{0}(x)$ and $\overline{f}_{1}(x)$  involved in  \eqref{9.4}, we know that $\deg( \varpi_{1,1}(x) ) \leq 2 $ and  $\deg(\varpi_{1,2}(x)) \leq 2$.

The detailed computation of these latest polynomials (confirmed by the use of a computer algebra software) allows us to conclude that
\newline $\deg( \varpi_{1,1}(x) ) \leq 1 $ and  $\deg(\varpi_{1,2}(x)) \leq 1$ if and only if $a_{2}^{[2]}=0$ and $\alpha_{4}= \frac{\alpha_{2}\gamma_{3}}{\gamma_{2}}$, respectively.

Subsequently, we add to the set of hypotheses of this discussion the conditions:
\begin{center}$a_{2}^{[2]}=0$ and $\alpha_{4}= \dfrac{\alpha_{2}\gamma_{3}}{\gamma_{2}}.$\end{center}

Let us now look at \eqref{Dcomplete} of Lemma \ref{basic-lemma}:
$$D \Big( \overline{p}_{0}(x)u_{0}+\overline{p}_{1}(x)u_{1} \Big)+ \left(  2a_{1}(x)+4p_{0}(x) \right) u_{0} + 4p_{1}(x) u_{1} = 0 $$
taking into account the information considered up to this moment, namely, $a_{3}(x)= \tau a_{2}(x)$ (thus $a_{3}^{[3]}=0$) , $a_{1}^{[1]} \neq 0$, $a_{2}^{[2]}=0$  and $\alpha_{4}=\frac{\alpha_{2}\gamma_{3}}{\gamma_{2}}.$

Looking carefully at the term $ \left(  2a_{1}(x)+4p_{0}(x) \right) u_{0}$ we conclude that 
$$ 2a_{1}(x)+4p_{0}(x) = 2E_{1}(x) \Leftrightarrow a_{1}(x)= -\frac{1}{3\gamma_{1}}(x-\beta_{0}).$$

Hence, considering this choice of coefficient $a_{1}(x)$, identity \eqref{Dcomplete} has the form
$$D \Big( \overline{p}_{0}(x)u_{0}+\overline{p}_{1}(x)u_{1} \Big)+ 2E_{1}(x) u_{0} + \frac{4}{3}A_{0}(x)u_{1} = 0 $$
or $D \Big( \overline{p}_{0}(x)u_{0}+\overline{p}_{1}(x)u_{1} \Big)+ 2E_{1}(x) u_{0} +2A_{0}(x)u_{1}-\frac{2}{3}A_{0}(x)u_{1} = 0 $.

Replacing the final term $-\frac{2}{3}A_{0}(x)u_{1}$ by the information provided by \eqref{Eq-case2-1}, specifically  $\frac{2}{3}A_{0}(x)D\Big( \varpi_{1,1}(x)u_0+ \varpi_{1,2}(x)u_1  \Big) $, we get the following.
\begin{align}
&D\Big(  \varpi_{2,1}(x) u_{0} + \varpi_{2,2}(x) u_{1} \Big) +2E_{1}(x)u_{0}+2A_{0}(x)u_{1}=0\\
&\label{varpi21} \varpi_{2,1}(x)= \overline{p}_{0}(x)+\frac{2}{3}A_{0}(x)\varpi_{1,1}(x) \\
&\label{varpi22} \varpi_{2,2}(x)= \overline{p}_{1}(x)+\frac{2}{3}A_{0}(x)\varpi_{1,2}(x) 
\end{align}

In brief, we can assert the next result.

\begin{theorem}\label{a3=Ca2}
 Let  $\{P_{n}(x)\}_{n \geq 0}$ be a $2$-orthogonal MPS fulfilling $\alpha_{4}= \dfrac{\alpha_{2}\gamma_{3}}{\gamma_{2}}$ and $J\left( P_{n} (x) \right)= \lambda_{n}^{[0]}P_{n}(x), \; n \geq 0$,  with $J$ defined by \eqref{third-order-J}.
 If the polynomial coefficients of operator $J$ fulfil:
\begin{align*}
&a_{3}(x)=\tau\, a_{2}(x)\, , \textrm{for some\;} \tau \neq 0 \;\; , \;\; a_{1}(x)= -\frac{1}{3\gamma_{1}}(x-\beta_{0}) \; \; , \;\;\deg(a_{2}(x)) \leq 1\, , \\
& a^{[2]}_{1} \neq \dfrac{2\tau}{\gamma_{1}(m+1)}-\dfrac{2}{3\gamma_1}\left(\beta_{1}-\beta_{3}  \right) , m\geq 0, 
\end{align*}
then $\{P_{n}(x)\}_{n \geq 0}$ is (Hahn) classical and the regular vector functional $U=(u_{0}, u_{1})^{T}$ fulfils
$$D\left(\Phi(x) U \right) + \Psi(x) U = 0,$$
\begin{tabular}{cc}
$\quad\quad$ $\Psi(x)=\left[\begin{array}{cc}0 & 1 \\ 2E_{1}(x) & 0 \end{array}\right], $
 & 
$\Phi(x)= \left[\begin{array}{cc}\varpi_{1,1}(x) & \varpi_{1,2}(x) \\  \varpi_{2,1}(x) & \varpi_{2,2}(x)\end{array}\right]$,
 \end{tabular}
\noindent with entries $\varpi_{i,j}(x)$ defined by \eqref{varpi11}, \eqref{varpi12}, \eqref{varpi21} and \eqref{varpi22}, that under the conditions described have the expressions indicated in Table \ref{tab:table-1}.
\end{theorem}
\begin{proof}
The arguments presented in this section demonstrate this result, except for the need of the condition  
$$a^{[2]}_{1} \neq \dfrac{2\tau}{\gamma_{1}(m+1)}-\dfrac{2}{3\gamma_1}\left(\beta_{1}-\beta_{3}  \right) , m\geq 0\, .$$

As mention in the previous section, the characterisation of the classical $d$-orthogonal polynomial sequences (in Hahn's sense) provided by \cite{d-classical}  imposes that the coefficient of $x^2$ in $\varpi_{2,1}(x)$ is different from $\frac{2}{\gamma_{1}(m+1)}$, $m \geq 0$, and the coefficient of $x$ in $\varpi_{1,2}(x)$ is different from $\frac{1}{m+1}$, $m \geq 0$ (cf. p.183 of \cite{d-classical}). The first requirement is naturally assured since $\deg\left(\varpi_{2,1}(x)\right) \leq 1$ and from the second we get the above inequality, since the coefficient of $x$ in $\varpi_{1,2}(x)$ is given by $\frac{2(\beta_1-\beta_3)+3\gamma_1a^{[2]}_1}{6\tau}$.
\end{proof}

\begin{landscape}
\begin{table}[H]
\begin{footnotesize}
\begin{align*}
&\varpi_{1,1}(x)= \frac{3 \beta _0 \gamma _2+\alpha _2 \beta _0 \left(\beta _1+\beta _2-2 \left(\beta _3+\tau \right)\right)+\gamma _1 \left(-2 \beta _0-3 \beta _1+2 \beta _3+6 \tau \right)}{6 \gamma _1 \tau}\\
&+\frac{x \left(3 \left(\gamma _1-\gamma _2\right)-\alpha _2 \left(\beta _1+\beta _2-2 \left(\beta _3+\tau\right)\right)\right)}{6 \gamma _1 \tau}\\
& \varpi_{1,2}(x)= \frac{\gamma _1 \left(3 \gamma _1 a^{[2]}_{0}-2 \left(\alpha _2+2 \beta _0 \tau+\beta _1 \left(\beta _1-\beta _3-2 \tau\right)\right)\right)+\alpha _1 \left(-\gamma _1+3 \gamma _2+\alpha _2 \left(\beta _1+\beta _2-2 \left(\beta _3+\tau\right)\right)\right)}{6 \gamma _1 \tau}\\
&+\frac{x \left(3 \gamma _1 a^{[2]}_{1}+2 \beta _1-2 \beta _3\right)}{6 \tau}\\
& \varpi_{2,1}(x)= \frac{-3 \alpha _1 \beta _0 \gamma _2^2+\gamma _2 \gamma _1 \left(\alpha _1 \left(2 \beta _0-2 \beta _3+3 \left(\beta _1+\tau\right)\right)+6 \beta _0 \left(\beta _0-\beta _2\right) \tau\right)+\alpha _2 \beta _0 \left(3 \alpha _3 \gamma _1 \tau-\alpha _1 \gamma _2 \left(\beta _1+\beta _2-2 \beta _3+\tau\right)\right)-3 \alpha _3 \gamma _1^2 \tau}{9 \gamma _1^2 \gamma _2 \tau}\\
&+\frac{x \left(\alpha _2 \left(\alpha _1 \gamma _2 \left(\beta _1+\beta _2-2 \beta _3+\tau\right)-3 \alpha _3 \gamma _1 \tau\right)+3 \gamma _2 \left(\alpha _1 \left(\gamma _2-\gamma _1\right)+2 \left(\beta _2-\beta _0\right) \gamma _1 \tau\right)\right)}{9 \gamma _1^2 \gamma _2 \tau}\\
& \varpi_{2,2}(x)=\frac{\alpha _1 \gamma _1 \left(\gamma _2 \left(-3 \gamma _1 a^{[2]}_{0}+7 \beta _0 \tau+2 \left(\beta _1^2+\beta _1 \left(\tau-\beta _3\right)-3 \beta _2 \tau\right)\right)+\alpha _2 \left(2 \gamma _2+3 \alpha _3 \tau\right)\right)+\alpha _1^2 \left(-\gamma _2\right) \left(-\gamma _1+3 \gamma _2+\alpha _2 \left(\beta _1+\beta _2-2 \beta _3+\tau\right)\right)-3 \gamma _1^2 \tau \left(\alpha _3 \beta _1-3 \gamma _2\right)}{9 \gamma _1^2 \gamma _2 \tau}\\
&+\frac{x \left(3 \alpha _3 \gamma _1 \tau-\alpha _1 \gamma _2 \left(3 \gamma _1 a^{[2]}_{1}+2 \beta _1-2 \beta _3+3 \tau\right)\right)}{9 \gamma _1 \gamma _2 \tau}
\end{align*}
\end{footnotesize}
\caption{\label{tab:table-1} Polynomial entries of matrix $\Phi(x)$ of Theorem \ref{a3=Ca2}.}
\end{table}
\end{landscape}
\section*{Acknowledgements}

\noindent  This work was partially supported by
 CMUP (UIDB/00144/2020), which is funded by FCT (Portugal).


\end{document}